\documentclass[a4paper,12pt] {article}

\usepackage[english]{babel}
\usepackage{amsmath,amssymb,amsthm,amsopn}
\usepackage{yfonts}
\usepackage{float}
\usepackage{tikz}
\usepackage{fancyhdr}
\usepackage{graphicx}
\usepackage{mathabx}
\usepackage{bm}

\newtheorem{Th}{Theorem}
\newtheorem{Lem}{Lemma}
\newtheorem{Prop}{Proposition}
\newtheorem{Cor}{Corollary}
\newtheorem{Def}{Definition}
\newtheorem{Hyp}{Hypothesis}
\newtheorem{Rmk}{Remark}

\def\bbN{\mathbb{N}}
\def\bbR{\mathbb{R}}
\def\bbS{\mathbb{S}}
\def\bbZ{\mathbb{Z}}
\def\rmex{{\mathrm{ex}}}
\def\rmd{{\mathrm{d}}} 
\def\d{{\mathrm{d}}} 
\def\rmZ{{{\tiny\mathrm{Z}}}} 

\def\cB{\mathcal{B}}
\def\cC{\mathcal{C}}
\def\cD{\mathcal{D}}
\def\cE{\mathcal{E}}
\def\cH{\mathcal{H}}
\def\cL{\mathcal{L}}

\def\cS{\mathcal{S}}
\def\cT{\mathcal{T}}

\def\gH{W^\mathfrak{h}}

\def\bmmu{{\bm \mu}}

\def\div{\operatorname{div}}
\def\rot{\operatorname{rot}}
\def\supp{\operatorname{supp}}
\def\diam{\operatorname{diam}}

\def\eps{\varepsilon}

\title{Static ferromagnetic materials:\\ from the microscopic to the mesoscopic scale}
\author{\normalsize Brigitte Bid\'egaray-Fesquet, Quentin Jouet, and St\'ephane Labb\'e\\ 
\small Laboratoire Jean Kuntzmann and Universit\'e de Grenoble, \\
\small BP 53, 38041 Grenoble Cedex 1, France}
\date{\today}

\begin{document}

\maketitle

\begin{abstract}
Thanks to averaging processes and $\Gamma$-convergence techniques, we are able to link a microscopic description of ferromagnetic materials based on  spin lattices and their mesoscopic description in the static framework for the three fundamental contributions: exchange, magnetostatics and external.The results are in accordance with the classical continuous description of ferromagnetic phenomena and justifies it. This work is a seed towards a dynamic description of ferromagnetic materials.
\end{abstract}

\section{Introduction}

The continuous description of ferromagnetic materials has been introduced since the 1960 {\it via} the micromagnetism model developed by Brown \cite{Brown63}. 
This model, based on a thermodynamical description of ferromagnetic phenomena, has proved its efficiency in numerous works {\it via} relevant simulations (\cite{proc072215,Brown.LaBonte:Structure,Fischer:grain.size,Vukad:Influence}). Nevertheless, several problems persist in the description, from the thermic effects to the magnetostrictive behaviors. 
These problems are very sharp and, in order to understand their modeling, one need to understand the microscopic behavior of magnetization (atomic scale) and the link between this scale and the mesoscopic scale (continuous magnetic matter scale of ferromagnetic effects). 
In this paper, we focus on the beginning of this program: the link between a microscopic description of ferromagnetic materials and their mesoscopic description in the static framework for the three fundamental contributions: exchange, magnetostatic and external field \cite{Aharoni:introduc,MR2060431,Labbe.Bertin:Microwave}. 
Here, the microscopic scale designates the atomic scale where atom nuclei are assumed to be pointwise electric charges bearing one magnetic moment induced by the atom electronic cloud; the mesoscopic scale designates the continuous description of matter for which the ferromagnetic effects are significative. One of the first formal approach of the link between the mesoscopic behavior of ferromagnetic materials and a microscopic thermodynamic  description of this behavior can be found in \cite{Brown66} for a model focused on the magnetoelastic phenomenon. In this book, the authors gives a formal approach of the asymptotics.

We can cite also several mathematical studies connected to the problem we study in this paper. First, in \cite{MR1308877}, authors study the convergence of energies for given sequences of magnetizations and various oscillatory behaviors of the magnetization. This approach does not consider the convergence of minimizers but of energies and shows the importance of controlling local oscillations of the magnetization in order to obtain a micromagnetic model which ensures that the modulus of the limit magnetization is constant on the ferromagnetic domain. We can also mention the study performed in \cite{MR2186037} who emphasizes the importance of controlling local oscillations of the magnetization in order to compute the limit magnetostatic interactions. This study is a derivation of the interaction forces between two magnetic bodies for given magnetic configurations built from the microscopic interactions between magnetic moments. It gives interesting information about the magnetostatic field of a body depending on the local regularity of the magnetization. The study of the limit mechanical stress is also studied in \cite{MR1423008} where the authors build a theory for the description of deformable magnetic bodies. A micro--meso approach is adopted in order to justify, from microscopic well interactions of dipoles, the interaction at the meso-scale, but this does not tackle the problem of the limit of minimizers. Nevertheless, in \cite{MR2505362}, a study of the minimizer convergence is performed in the case of Ginzburg--Landau systems. Authors focus on a $\Gamma$-convergence with a Dirichlet energy where the limit constraint is the divergence-free condition and not the constraint on the local modulus of the obtained field. In this case, the constraint on the modulus derives from a penalization term and not from the modeling hypothesis. 

In our paper, we are interested in the convergence of minimizer sequences of the microscopic description of ferromagnetic materials toward a continuous micromagnetic description which respects the constraint of uniformity of the local modulus of the magnetization. This constraint is obtained by a hypothesis on the local variations of the magnetization. In particular, this control of the local variation induced by the Heisenberg interaction model has to be tempered with a hypothesis on allowed defects. Forgetting defects would give rise to a constrained system and in particular limits in $H^s(\Omega)$ for $s$ strictly greater than $1$. Such a regularity would not allow key microstructures of the micromagnetism theory: the vortices. 

At the microscopic scale, we describe the material as a regular periodic spin lattice intersected with the magnetic domain (a bounded open set of $\bbR^3$). Section \ref{sec_lattice} is dedicated to the mathematical description of the microscopic model (the spin lattice) and to an averaging process towards a mesoscopic model which leads to a constant norm magnetic field, in accordance with usual models of micromagnetism. The microscopic energies are introduced and several modeling hypotheses are set. The main hypothesis is induced by the adiabatic behavior of the Heisenberg energy compared to the global ferromagnetic energy. This hypothesis gives a constraint on neighboring magnetic moments. In fact, this constraint is verified by a set of minimizers for a given lattice. 

The energy induced by the Heisenberg interaction is more difficult to treat. Section \ref{sec_exchange} addresses the study of this contribution and, in particular, its asymptotic behavior for sequences of lattices verifying the modeling hypothesis. This hypothesis ensures compactness which allows to use $\Gamma$-convergence tools in $H^1$. The limiting energy constructed from the discrete magnetization is the exchange energy (Theorem \ref{th_main}). 

In Section \ref{sec_total}, we introduce the demagnetization and external energies both for discrete lattices and the continuous model and finally obtain a convergence result for the sum of the Heisenberg, demagnetization, and external contributions (Theorem \ref{th_total}).


\section{Mathematical descriptions of a spin lattice}
\label{sec_lattice}

\subsection{Atomic lattice description}

We consider a collection of spins which are located on the nodes of a periodic lattice $\cL$ in the $\bbR^d$ space ($d=1,2,3$) with mesh size $a>0$.
In the scope of this paper we will restrict to the case of 1D, square or cubic lattices, $\cL$ is simply $a\bbZ^d$, but we can think of more complex lattices. 
Here all the nodes play the same role to ensure a unique definition of neighbors. 
In the opposite case a multi-species model should be used. \\
The nodes are indexed by $i\in\bbN$ and we denote by $x_i$ the $i$th spin location and $\mu_{x_i}$ the corresponding spin value (magnetic moment). 
The norm of these magnetic moments are scaled to the unit value and therefore for all $i\in\bbN$, $\mu_{x_i}\in\bbS^2$, where $\bbS^2$ is the unit sphere of $\bbR^3$.

Instead of describing a collection of magnetic moments, we can gather all the values in one single vector field $\mu$ defined by
\begin{equation*}
\forall x\in \bbR^d ,\;\; \mu (x) = \sum_{i\in \bbN} \mu_{x_i} \delta_{x_i} (x),
\end{equation*}
where $\delta_{x_i}$ is the Dirac delta function centered at $x_i$.

\subsection{Scaling}

We want to obtain an homogenized model of the spin lattice, i.e. give a description when this lattice is seen from far. 
Instead of really doing this, we will perform some dual transformation, i.e. consider only nodes that are included in some fixed bounded domain $\Omega$, and shrink the lattice (as shown on Figure \ref{fig_shrink} for $d=2$).
More precisely, we suppose that $0\in\mathring\Omega$ and for all $n\in\bbN^*$, using the homothety $h_n(x):= x/n$, $\forall x \in \bbR^d$, we define
\begin{itemize}
\item $\cL_n = h_n(\cL)$, the shrunk lattice;
\item $\cL_{n,\Omega} = \cL_n\cap\Omega$, the nodes of the shrunk lattice that belong to $\Omega$;
\item $\displaystyle \mu_n \in (\bbS^2)^{\cL_{n,\Omega}}$, the shrunk vector field.  
\end{itemize}
We notice that for all $y\in\bbR^d$, 
\begin{equation*}
\mu_n(y) = \sum_{x\in\cL_{n,\Omega}}\mu_{n,x}\delta_x(y),
\end{equation*}
where $\mu_{n,x}=\mu_{h_n^{-1}(x)}$. 
We assume that $\Omega$ has a sufficiently regular boundary in order that the number of nodes belonging to $\cL_{n,\Omega}$ is 
\begin{equation*}
\#\cL_{n,\Omega} = C n^d \left(1+O\left(\frac1n\right)\right).
\end{equation*}
where $C$ is a constant which only depends on $\cL$, $a$ and $\Omega$ (which are constants of our problem).

\begin{figure}[H]
\begin{center}
\begin{tikzpicture}[scale=.4]
\foreach \x in {0,1,...,10}
  \foreach \y in {0,1,...,10}
    {
     \filldraw[xshift=\x cm,yshift=\y cm] (0,0) circle (.1 cm);
     }
\draw[thick,rounded corners=8pt] (3,3) -- (7,2) -- (9,3) -- (7,6) -- (8,8) -- (6,7) -- (3,8) -- (2,5) -- cycle;
\draw (5,-1) node {(a)};

\foreach \x in {0,0.5,...,10}
  \foreach \y in {0,0.5,...,10}
     {
     \draw[xshift=12cm+\x cm,yshift=\y cm] (0,0) circle (.03 cm);
     }
\draw[thick,rounded corners=8pt,xshift=12cm] (3,3) -- (7,2) -- (9,3) -- (7,6) -- (8,8) -- (6,7) -- (3,8) -- (2,5) -- cycle;
\draw (17,-1) node {(b)};

\shadedraw[thick,rounded corners=8pt,shading=radial,xshift=24cm] (3,3) -- (7,2) -- (9,3) -- (7,6) -- (8,8) -- (6,7) -- (3,8) -- (2,5) -- cycle;
\draw (29,-1) node {(c)};
\end{tikzpicture}
\end{center}
\caption{\label{fig_shrink}Scaling of a 2D spin lattice: each sub-plot represents the domain $\Omega$ and (a) the square lattice $\cL$; (b) the shrunk square lattice $\cL_2$; (c) the homogeneized lattice in $\Omega$ as $n\to\infty$.}
\end{figure}
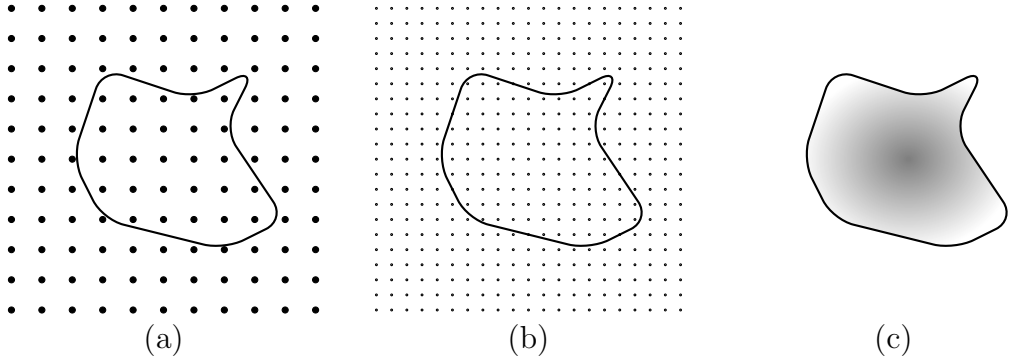

\subsection{Regularity assumptions}

In order to pass to the limit as $n\to\infty$, we have to assume that the magnetic moments are locally almost aligned. 
The definition of locality is given by an integer multiple $k\in\bbN^*$ of the shrunk mesh size $a/n$.

We define a first regularity assumption that only depends on the distance.
For all $x\in\bbR^d$ and $r>0$, we denote by $\cB(x,r)$ the ball of center $x$ and radius $r$ in $\bbR^d$. 

\begin{Hyp}
\label{hyp_local}
For all $n\in \bbN^*$, there exists $\zeta_n>0$ such that 
\begin{equation*}
\forall x \in \Omega,\ \forall y,z\in\cL_{n,\Omega}\cap\cB\left(x,\frac{ka}n\right),
\hspace{1cm} 
1-\zeta_n\leq\mu_{n,y}\cdot\mu_{n,z}\leq 1,
\end{equation*}
where $\zeta_n=O(1/n^2)$.
\end{Hyp}

We are indeed interested in the limit when we have a more and more dense lattice of more and more aligned spins. We therefore perform a diagonal process and correlate $n$ and the  spin alignment.

To define the averaging process we will also need to assume the $\Omega$ has the \textit{uniform cone property}.

\begin{Hyp}
\label{hyp_cone}
There exists an angle $\alpha$ and a radius $r$, such that for all $y\in\Omega$ there exists a cone $C_y$ of angle $\alpha$ with vertex at $y$ such that $B(y,r)\cap C_y\subset \Omega$.
\end{Hyp}

\subsection{Partitions adapted to the lattices}

Let us first work on the initial lattice.
To this aim, we define a partition of unity $(\rho_x)_{x\in k\cL}$ adapted to the dilated lattice $k\cL$.
Since our lattice is uniform and all the nodes are equivalent, all the $\rho_x$ are equal up to a translation (see Figure \ref{fig_partition_kL}), i.e. there exists $\rho^\star\in\cC_0^\infty(\bbR^d;\bbR)$ such that
\begin{equation*}
\forall x\in k\cL,\ \forall y\in\bbR^d,\ \rho_x(y)=\rho^\star(y-x).
\end{equation*}

\begin{figure}[H]
\begin{center}
\begin{tikzpicture}[scale=.7]
\draw (0,0) -- (17,0);
\foreach \x in {.5,1.5,...,16.5}
     \filldraw[xshift=\x cm] (0,0) circle (.1 cm);
\foreach \x in {0,5,10}
     \draw[xshift=.5cm+\x cm] (0,0) .. controls (.5,0) and (.5,2) .. (1,2) -- (5,2) .. controls (5.5,2) and (5.5,0) .. (6,0);
\draw[<->] (3.5,-.5) -- node[below] {$a$}(4.5,-.5);
\draw[<->] (5.5,-.5) -- node[below] {$ka$}(10.5,-.5);
\end{tikzpicture}
\end{center}
\caption{\label{fig_partition_kL}Partition $(\rho_x)_{x\in k\cL}$ in dimension 1.}
\end{figure}
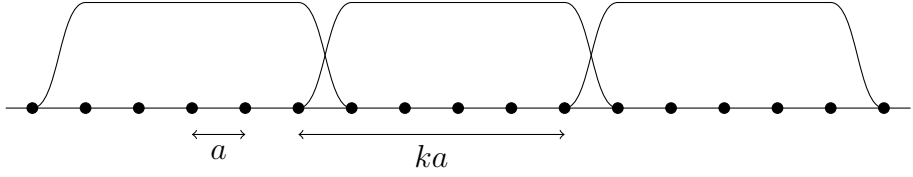

Without lack of generality, we assume that $\rho^\star>0$ and $\supp\rho^\star\in B(ka)$.
By definition of a partition
\begin{equation*}
\forall y\in\bbR^d,\ \sum_{x\in k\cL} \rho_x(y) = \sum_{x\in k\cL} \rho^\star(y-x) = 1.
\end{equation*}
Let $n_k$ be the number of nodes in $\cB(0,ka)$, which is e.g. $n_k\le(2k+1)^d$ for a cubic lattice (number of nodes contained in the cube around the sphere).
If we now sum over all the $x\in\cL$, we have the same translation property and 
\begin{equation*}
\forall y\in\bbR^d,\ \sum_{x\in \cL} \rho^\star(y-x) = n_k.
\end{equation*}
Defining $\rho=n_k^{-1}\rho^{\star}$, we have 
\begin{equation*}
\forall y\in\bbR^d,\ \sum_{x\in\cL} \rho(y-x) = 1,
\end{equation*}
and therefore a continuous partition of unity on $\bbR^d$.\\

Now, we adapt this partition to the shrunk lattice $\cL_n$.
Defining
\begin{equation*}
\forall n\in\bbN^*,\ \rho_n(x)=\rho(h_n^{-1}x),
\end{equation*}
we have a continuous partition adapted to the shrunk lattice (see Figure \ref{fig_partition_x})
\begin{equation*}
\forall y\in\bbR^d,\ \sum_{x\in\cL_n}\rho_n(y-x)=1 \textrm{ and } \supp\rho_n\subset\cB\left(0,\frac{ka}n\right).
\end{equation*}

\begin{figure}[H]
\begin{center}
\begin{tikzpicture}[scale=.7]
\draw (0,0) -- (17,0);
\foreach \x in {.5,1.5,...,16.5}
     \filldraw[xshift=\x cm] (0,0) circle (.1 cm);
\foreach \x in {0,5,10}
     \draw[xshift=.5cm+\x cm] (0,0) .. controls (.5,0) and (.5,2) .. (1,2) -- (5,2) .. controls (5.5,2) and (5.5,0) .. (6,0);
\foreach \x in {0,1,...,10}
     \draw[style=dotted,xshift=.5cm+\x cm] (0,0) .. controls (.5,0) and (.5,2) .. (1,2) -- (5,2) .. controls (5.5,2) and (5.5,0) .. (6,0);
\draw[<->] (3.5,-.5) -- node[below] {$a/n$}(4.5,-.5);
\draw[<->] (5.5,-.5) -- node[below] {$ka/n$}(10.5,-.5);
\end{tikzpicture}
\end{center}
\caption{\label{fig_partition_x}Partition $(\rho_x)_{x\in\cL_n}$ in dimension $d=1$.}
\end{figure}
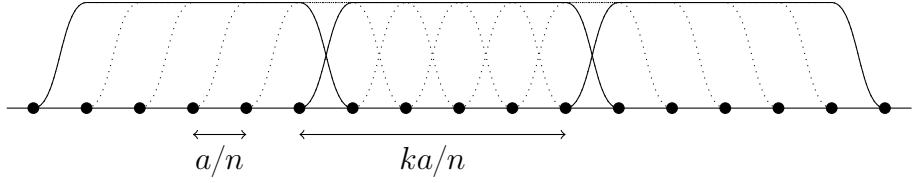

Since $\nabla \rho^\star$ is uniformly bounded (i.e. $O(1)$), then $\nabla \rho_n$ is uniformly $O(n)$. \\

To define an averaging process we will have to use a truncated partition of unity, namely
\begin{equation*}
\Phi_n(y)=\sum_{x\in\cL_{n,\Omega}}\rho_n(y-x).
\end{equation*}
We clearly have $0\leq\Phi_n(y)\leq1$ and as a finite sum of $\cC^\infty$ functions, $\Phi_n\in\cC^\infty(\bbR^d;\bbR)$.
We also have a stronger result, namely $\Phi_n(y)$ is bounded from below uniformly in $n$ and $y\in\Omega$: there exists $b>0$ and $n_0\in\bbN^*$ such that 
\begin{equation*}
\forall y\in\Omega,\ \forall n\geq n_0,
\hspace{1cm}
b\leq\Phi_n(y)\leq1.
\end{equation*}
This result stems from the "cone property" (Hypothesis \ref{hyp_cone}).
Besides $\nabla\Phi_n$ is uniformly $O(n)$.

\subsection{Averaging process}

From the sequence $(\mu_n)_{n\in \bbN^*}$ of vector fields, we now define a new sequence $(m_n)_{n\geq n_0}$ of regular vector fields on $\Omega$: for all $n\geq n_0$, we define
\begin{equation*}
m_n = \frac1{\Phi_n}\mu_n\star\rho_n,
\end{equation*}

\begin{Rmk} 
Since $\rho_n$ and $\Phi_n\in\cC^\infty(\bbR^d;\bbR)$, and $\Phi_n$ is bounded from below by $b>0$, it immediately follows that $m_n\in \cC^\infty(\Omega;\bbR^3)$. 
In particular, since $\Omega$ is bounded, $m_n\in H^1(\Omega;\bbR^3)$.
\end{Rmk}

\begin{Lem} 
\label{lem_borneL2}
The sequence $(m_n)_{n\geq n_0}$ is bounded in $L^\infty(\Omega;\bbR^3)$ and $L^2(\Omega;\bbR^3)$. 
\end{Lem}

\begin{proof}
Recall
\begin{equation*}
\mu_n = \sum_{x\in\cL_{n,\Omega}}\mu_{n,x}\delta_x,
\end{equation*}
hence, for all $y\in \Omega$,
\begin{equation*}
m_n(y) =  \frac1{\Phi_n(y)}\sum_{x\in\cL_{n,\Omega}} \mu_{n,x} \rho_n(y-x).
\end{equation*}
We clearly have 
\begin{eqnarray*}
\left| \frac1{\Phi_n(y)}\sum_{x\in\cL_{n,\Omega}} \mu_{n,x} \rho_n(y-x)\right|
& \leq &  \frac1{\Phi_n(y)}\sum_{x\in\cL_{n,\Omega}} |\mu_{n,x}| \rho_n(y-x) = \frac{\Phi_n(y)}{\Phi_n(y)}= 1, \\
\|m_n\|_{L^\infty(\Omega;\bbR^3)} & \leq & 1, \\ 
\|m_n\|_{L^2(\Omega;\bbR^3)} & \leq & \sqrt{|\Omega|}. 
\end{eqnarray*}

\end{proof}

\begin{Prop} 
\label{prop_borneH1}
Under Hypothesis \ref{hyp_local}, $(m_n)_{n\geq n_0}$ is a bounded sequence in $H^1(\Omega;\bbR^3)$. 
\end{Prop}

\begin{proof}
For all $y\in \Omega$, we write
\begin{eqnarray*}
\nabla m_n(y) &=& \nabla \left( \frac1{\Phi_n(y)} \right)\sum_{x\in\cL_{n,\Omega}} \mu_{n,x} \rho_n(y-x)+  \frac1{\Phi_n(y)} \nabla \!\!\sum_{x\in\cL_{n,\Omega}} \mu_{n,x} \rho_n(y-x),\\
&=& -\frac{\nabla \Phi_n(y)}{\Phi_n(y)} \otimes m_n(y) + \frac1{\Phi_n(y)} \sum_{x\in\cL_{n,\Omega}} \nabla \rho_n(y-x) \otimes \mu_{n,x}.
\end{eqnarray*}
Let $x_y \in \cL_{n,\Omega} \cap \supp \rho_n (y-\cdot)$, we can write
\begin{eqnarray*}
\sum_{x\in\cL_{n,\Omega}} \nabla \rho_n(y-x) \otimes \mu_{n,x}
& = & \sum_{x\in\cL_{n,\Omega}} \nabla \rho_n(y-x) \otimes (\mu_{n,x_y}+(\mu_{n,x}-\mu_{n,x_y}))\\
& = & \sum_{x\in\cL_{n,\Omega}} \nabla \rho_n(y-x) \otimes \mu_{n,x_y} \\
&& + \sum_{x\in\cL_{n,\Omega}} \nabla \rho_n(y-x) \otimes (\mu_{n,x}-\mu_{n,x_y})\\
& = & \nabla\Phi_n(y) \otimes \mu_{n,x_y} \\
&& + \sum_{x\in\cL_{n,\Omega}} \nabla \rho_n(y-x) \otimes (\mu_{n,x}-\mu_{n,x_y}).
\end{eqnarray*}
Hence
\begin{eqnarray*}
\nabla m_n(y) 
& = & -\frac{\nabla \Phi_n(y)}{\Phi_n(y)} \otimes m_n(y) +\frac{\nabla \Phi_n(y)}{\Phi_n(y)} \otimes \mu_{n,x_y} \\ 
& + & \left( \frac1{\Phi_n(y)} \right)\sum_{x\in\cL_{n,\Omega}} \nabla \rho_n(y-x) \otimes (\mu_{n,x}-\mu_{n,x_y}) \allowdisplaybreaks \\
& = & -\frac{\nabla \Phi_n(y)}{\Phi_n(y)} \otimes \left(\mu_{n,x_y}+\frac1{\Phi_n(y)}\sum_{x\in\cL_{n,\Omega}}(\mu_{n,x}-\mu_{n,x_y})\rho_n(y-x)\right) \\ 
& & + \frac{\nabla \Phi_n(y)}{\Phi_n(y)} \otimes \mu_{n,x_y} 
+ \left( \frac1{\Phi_n(y)} \right)\sum_{x\in\cL_{n,\Omega}}\nabla \rho_n(y-x) \otimes (\mu_{n,x}-\mu_{n,x_y}) \allowdisplaybreaks \\
& = & \frac{\nabla \Phi_n(y)}{\Phi_n(y)^2} \otimes \left(\sum_{x\in\cL_{n,\Omega}}(\mu_{n,x_y}-\mu_{n,x})\rho_n(y-x)\right) \\
& + & \left( \frac1{\Phi_n(y)} \right)\sum_{x\in\cL_{n,\Omega}}\nabla \rho_n(y-x) \otimes (\mu_{n,x}-\mu_{n,x_y}).
\end{eqnarray*}
In both sums there is only an $O(1)$ number of terms for which $\rho_n$ or $\nabla \rho_n$ is non-zero.
In the first sum, by Hypothesis \ref{hyp_local}, $|\mu_{n,x}-\mu_{n,x_y}| = O(\zeta_n^{1/2}) = O(1/n)$. 
Besides $|\nabla \Phi_n(y)|=O(n)$ and $\Phi_n$ is bounded from below. The first term is therefore $O(1)$.
In the second sum $|\mu_{n,x}-\mu_{n,x_y}| = O(1/n)$ and $\nabla \rho_n(y-x)=O(n)$. 
Hence there exists $C'>0$ such that, for all $y\in \Omega$ and $n\geq n_0$, $|\nabla m_n(y)| \leq C'$.

Finally, $\Omega$ being bounded, hence $(\|\nabla m_n\|_{L^2(\Omega;\bbR^{3d})})_{n\geq n_0}$ is a bounded sequence.
\end{proof}

\subsection{Asymptotics towards a mesoscopic model}

Since $H^1(\Omega;\bbR^3)$ is weakly compact, Proposition \ref{prop_borneH1} implies that 
there exists $m\in H^1(\Omega;\bbR^3)$ such that 
\begin{equation*}
m_n\rightharpoonup m\textrm{ weakly in }H^1(\Omega;\bbR^3).
\end{equation*}

From now on, we have to assume that $\Omega$ is compact (closed) and has a piecewise $\cC^1$ boundary, to ensure that this implies that this convergence is strong in $L^2(\Omega;\bbR^3)$.

\begin{Prop}
Under Hypothesis \ref{hyp_local}, $m$ has a constant norm equal to 1 almost everywhere on $\Omega$. 
\end{Prop}

\begin{Rmk}
We recover here a property of the magnetization field in Brown's model of micromagnetism \cite{Brown63}, where the constant norm is assumed.
\end{Rmk}

\begin{proof}

We have seen that for all $n\geq n_0$ and $y\in \Omega$, $|m_n(y)| \leq 1$.
We want to show that it is also bounded from below and pass to the limit.
For all $y\in\Omega$, since $|\mu_{n,x}|^2=1$,
\begin{eqnarray*} 
|m_n(y)|^2 & = & \left| \frac1{\Phi_n(y) }\sum_{x\in\cL_{n,\Omega}}\mu_{n,x}\rho_n(y-x)\right|^2 \\
& = & \frac1{\Phi_n(y)^2} \sum_{x,x'\in\cL_{n,\Omega}}\mu_{n,x}\cdot\mu_{n,x'}\ \rho_n(y-x)\rho_n(y-x').
\end{eqnarray*}
Since $\supp\rho_n\subset\cB(0,ka/n)$, the sum runs indeed on $\cL_{n,\Omega}\cap\cB(y,ka/n)$:
\begin{equation*} 
|m_n(y)|^2 = \frac1{\Phi_n(y)^2} \sum_{x,x'\in\cL_{n,\Omega}\cap\cB(y,\frac{ka}n)}\mu_{n,x}\cdot\mu_{n,x'}\ \rho_n(y-x)\rho_n(y-x').
\end{equation*}
Hypothesis \ref{hyp_local} implies that
\begin{equation*}
\forall x,x'\in\cL_{n,\Omega}\cap\cB\left(y,\frac{ka}n\right),\; 1-\zeta_n\leq\mu_{n,x}\cdot\mu_{n,x'}\leq1.
\end{equation*}
and we also have that 
\begin{equation*}
\forall y\in \Omega,\ \sum_{x,x'\in\cL_{n,\Omega}}\rho_n(y-x)\rho_n(y-x') = \Phi_n(y)^2,
\end{equation*}
therefore
\begin{eqnarray*}
\frac1{\Phi_n(y)^2}( 1-\zeta_n)  \Phi_n(y)^2 \leq &|m_n(y)|^2&\leq \frac1{\Phi_n(y)^2} \Phi_n(y)^2, \\
1-\zeta_n  \leq &|m_n(y)|^2 &\leq 1.
\end{eqnarray*}
In the $L^2(\Omega;\bbR^3)$ limit, we therefore have $|m|=\chi_{\Omega}$ almost everywhere.
\end{proof}

\section{Towards the exchange energy}
\label{sec_exchange}

\subsection{Heisenberg interaction}
\label{sec_Heisenberg}

The interaction of the spins is described by the Heisenberg energy, which can be written as follows:
\begin{equation*}
e_{x,y}=-\frac12 A_{x,y}(\mu_x \cdot \mu_y-1),
\end{equation*}
where $A_{x,y}\geq0$ only depends on the distance between $x$ and $y$. We make the assumption that each node $x$ only interacts with its neighbors $N_x$.

Let $N_0\subset\cL$ be the set of neighboring nodes in interaction with the node $(0,0,0)$ \textit{via} the Heisenberg energy.
In the shrunk lattice, we will restrict the computation of this energy to elements $x$ and $y$ in $\cL_{n,\Omega}$.
Since the lattice is homogeneous, the neighbors of any given node $x\in\cL_{n,\Omega}$ can be deduced from the definition of $N_0$: $N_{n,\Omega,x}=(\{x\}+h_n^{-1}(N_0))\cap \Omega$.
We also assume that
(a) $N_0$ (and hence $N_{n,\Omega,x}$) is a finite set;
(b) $N_0\subset\cB(0,ka)$ (i.e. $N_{n,\Omega,x}\subset\cB(x,ka/n)$).

We therefore define the node energy by
\begin{equation*}
e_x = \sum_{y\in N_x}e_{x,y} 
= -\frac12\sum_{y\in N_x} A_{x,y}(\mu_x\cdot\mu_y-1) .
\end{equation*}

In what follows we will restrict the study to 3D cubic lattices. 
Dimensions 1 and 2 are also possible to treat in the same way. 
The fact that the lattice is cubic is used in the explicit computations, but our proof may be extended to other classes of regular lattices.
We also consider as neighbors only the 6 closest ones on the cubic lattice, which are at the same distance. Since $A_{x,y}$ only depends on the distance, which is now equal for all the neighbors, we can set 
\begin{equation*}
A_{x,y} = \begin{cases}
A>0, & \textrm{ if } y\in N_x,\\
0, & \textrm{ else.}
\end{cases}
\end{equation*}

For all $n\in\bbN^*$, we define the exchange energy of the domain $\Omega$ associated to the spin distribution $\mu_n$ summing up the node energies 
\begin{equation*}
\cE_{n,\rmex}(\mu_n) = \sum_{x\in\cL_{n,\Omega}} e_x 
= \frac an \sum_{x\in\cL_{n,\Omega}} \sum_{y\in N_{n,\Omega,x}} A |\mu_{n,y}-\mu_{n,x}|^2 .
\end{equation*}

\subsection{Spaces and convergence}

We define the space sequence $(W_n)_{n\in\bbN^*}$ by $W_n = \left(\bbR^3\right)^{\cL_{n,\Omega}}$. In the sequel, we set $W=\prod_{n\in\bbN^*} W_n$.

\begin{Hyp}
\label{hyp_local2}
Let $\bmmu = (\mu_n)_{n\in\bbN^*} \in  \prod_{n\in\bbN^*} W_n$.  
There exists $c>0$ such that 
\begin{equation*}
\forall n \in\bbN^*,\ \forall x\in \cL_{n,\Omega},\ \forall y\in N_{n,\Omega,x},
\hspace{1cm} 
|\mu_{n,x}-\mu_{n,y}|^2 \leq \frac{c}{n^2},
\end{equation*} 
except possibly for a subset $l_{n,\Omega}\subset\cL_{n,\Omega}$ such that $\#(l_{n,\Omega}) = O(n)$, and there exists a sequence $(c_n)_{n\geq 1} \in \bbR^{\bbN^*}$ such that $c_n \xrightarrow[n\to\infty]{} 0$ and 
\begin{equation*}
\forall n \in\bbN^*,\ \forall x\in l_{n,\Omega},\ \forall y\in N_{n,\Omega,x},\ |\mu_{n,x}-\mu_{n,y}|^2 \leq c_n.
\end{equation*}
\end{Hyp}

More precisely, given the constant $c$, we can state that $x\in l_{n,\Omega}$ if there exists $y\in N_{n,\Omega,x}$ such that $|\mu_{n,x}-\mu_{n,y}|^2 > c/n^2$.
Hypothesis \ref{hyp_local2} is a little weaker than Hypothesis \ref{hyp_local}. 

We denote by $\gH\subset W$ the set of $\bmmu=(\mu_n)_{n\in\bbN^*}$ for which Hypothesis \ref{hyp_local2} holds and $\gH_1 \subset \gH$ the set of $\bmmu\in\gH$ for which $|\mu_{n,x}|=1$, $\forall n\in\bbN^*,\ \forall x\in \cL_{n,\Omega}$.

We want to define the convergence of elements of $W$ towards elements of the limit space $H^1(\Omega;\bbR^3)$.
To this aim we define a partition of the domain $\Omega\subset\bbR^3$ in tetrahedra, by groups of 5 tetrahedra (see Figure \ref{fig_tetra} and Section \ref{sec_elements}). 

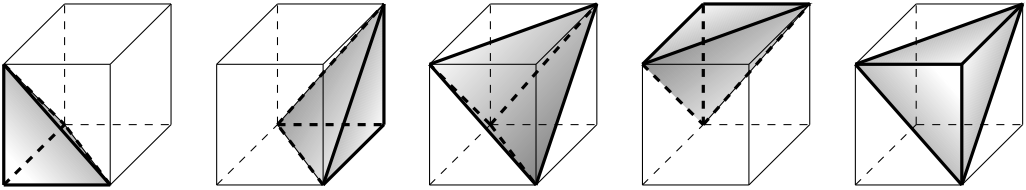
\begin{figure}[H]
\begin{center}
\begin{tikzpicture}[x=1.4cm,y=1.6cm,z=-.8cm]
\foreach \x in {0}
{
\shadedraw[shading=axis,shading angle=-45](\x,0) -- (\x+1,0) -- (\x,1);
\shadedraw[shading=axis,shading angle=-45](\x,1) -- (\x,0,-1) -- (\x+1,0);
\draw[very thick](\x+1,0) -- (\x,1);
\draw[very thick,dashed](\x+1,0) -- (\x,0,-1);
\draw[very thick,dashed](\x,1) -- (\x,0,-1);
\draw[very thick](\x,0) -- (\x+1,0);
\draw[very thick](\x,0) -- (\x,1);
\draw(\x,1) -- (\x+1,1);
\draw(\x+1,0) -- (\x+1,1);
\draw(\x,1) -- (\x,1,-1);
\draw(\x,1,-1) -- (\x+1,1,-1);
\draw(\x+1,1) -- (\x+1,1,-1);
\draw(\x+1,0) -- (\x+1,0,-1);
\draw(\x+1,0,-1) -- (\x+1,1,-1);
\draw[very thick,dashed](\x,0) -- (\x,0,-1);
\draw[dashed](\x,0,-1) -- (\x+1,0,-1);
\draw[dashed](\x,0,-1) -- (\x,1,-1);
}
\foreach \x in {2}
{
\shadedraw[shading=axis,shading angle=45](\x+1,0) -- (\x+1,0,-1) -- (\x+1,1,-1);
\shadedraw[shading=axis,shading angle=45](\x+1,0) -- (\x,0,-1) -- (\x+1,1,-1);
\draw[very thick](\x+1,0) -- (\x+1,1,-1);
\draw[very thick,dashed](\x+1,0) -- (\x,0,-1);
\draw[very thick,dashed](\x,0,-1) -- (\x+1,1,-1);
\draw(\x,0) -- (\x+1,0);
\draw(\x,0) -- (\x,1);
\draw(\x,1) -- (\x+1,1);
\draw(\x+1,0) -- (\x+1,1);
\draw(\x,1) -- (\x,1,-1);
\draw(\x,1,-1) -- (\x+1,1,-1);
\draw(\x+1,1) -- (\x+1,1,-1);
\draw[very thick](\x+1,0) -- (\x+1,0,-1);
\draw[very thick](\x+1,0,-1) -- (\x+1,1,-1);
\draw[dashed](\x,0) -- (\x,0,-1);
\draw[very thick,dashed](\x,0,-1) -- (\x+1,0,-1);
\draw[dashed](\x,0,-1) -- (\x,1,-1);
}
\foreach \x in {4}
{
\shadedraw[shading=axis,shading angle=-135](\x,0,-1) -- (\x,1) -- (\x+1,0);
\shadedraw[shading=axis,shading angle=-135](\x,1) -- (\x+1,1,-1) -- (\x+1,0);
\draw[very thick](\x+1,0) -- (\x+1,1,-1);
\draw[very thick](\x+1,0) -- (\x,1);
\draw[very thick](\x,1) -- (\x+1,1,-1);
\draw[very thick,dashed](\x,0,-1) -- (\x+1,1,-1);
\draw[very thick,dashed](\x,0,-1) -- (\x,1);
\draw[very thick,dashed](\x,0,-1) -- (\x+1,0);
\draw(\x,0) -- (\x+1,0);
\draw(\x,0) -- (\x,1);
\draw(\x,1) -- (\x+1,1);
\draw(\x+1,0) -- (\x+1,1);
\draw(\x,1) -- (\x,1,-1);
\draw(\x,1,-1) -- (\x+1,1,-1);
\draw(\x+1,1) -- (\x+1,1,-1);
\draw(\x+1,0) -- (\x+1,0,-1);
\draw(\x+1,0,-1) -- (\x+1,1,-1);
\draw[dashed](\x,0) -- (\x,0,-1);
\draw[dashed](\x,0,-1) -- (\x+1,0,-1);
\draw[dashed](\x,0,-1) -- (\x,1,-1);
}
\foreach \x in {6}
{
\shadedraw[shading=axis,shading angle=-135](\x,1,-1) -- (\x+1,1,-1) -- (\x,1);
\shadedraw[shading=axis,shading angle=45](\x,0,-1) -- (\x+1,1,-1) -- (\x,1);
\draw[very thick,dashed](\x,0,-1) -- (\x+1,1,-1);
\draw[very thick,dashed](\x,1) -- (\x,0,-1);
\draw[very thick](\x,1) -- (\x+1,1,-1);
\draw(\x,0) -- (\x+1,0);
\draw(\x,0) -- (\x,1);
\draw(\x,1) -- (\x+1,1);
\draw(\x+1,0) -- (\x+1,1);
\draw[very thick](\x,1) -- (\x,1,-1);
\draw[very thick](\x,1,-1) -- (\x+1,1,-1);
\draw(\x+1,1) -- (\x+1,1,-1);
\draw(\x+1,0) -- (\x+1,0,-1);
\draw(\x+1,0,-1) -- (\x+1,1,-1);
\draw[dashed](\x,0) -- (\x,0,-1);
\draw[dashed](\x,0,-1) -- (\x+1,0,-1);
\draw[very thick,dashed](\x,0,-1) -- (\x,1,-1);
}
\foreach \x in {8}
{
\shadedraw[shading=axis,shading angle=-135](\x+1,0) -- (\x+1,1) -- (\x+1,1,-1);
\shadedraw[shading=axis,shading angle=45](\x+1,1) -- (\x+1,1,-1) -- (\x,1);
\shadedraw[shading=axis,shading angle=135](\x+1,0) -- (\x+1,1) -- (\x,1);
\draw[very thick](\x+1,0) -- (\x+1,1,-1);
\draw[very thick](\x+1,0) -- (\x,1);
\draw[very thick](\x,1) -- (\x+1,1,-1);
\draw(\x,0) -- (\x+1,0);
\draw(\x,0) -- (\x,1);
\draw[very thick](\x,1) -- (\x+1,1);
\draw[very thick](\x+1,0) -- (\x+1,1);
\draw(\x,1) -- (\x,1,-1);
\draw(\x,1,-1) -- (\x+1,1,-1);
\draw[very thick](\x+1,1) -- (\x+1,1,-1);
\draw(\x+1,0) -- (\x+1,0,-1);
\draw(\x+1,0,-1) -- (\x+1,1,-1);
\draw[dashed](\x,0) -- (\x,0,-1);
\draw[dashed](\x,0,-1) -- (\x+1,0,-1);
\draw[dashed](\x,0,-1) -- (\x,1,-1);
}
\end{tikzpicture}

\caption{\label{fig_tetra}Partition of a single 3D lattice cell. These five tetrahedra are elements $\cT_n$, and $T_n$ for the corner tetrahedra, and $T_n^*$ for the center tetrahedron (see Section \ref{sec_elements}).}
\end{center}
\end{figure}

\begin{Def}
\label{def_Pn}
We define the projection $P_n: W_n \longrightarrow H^1(\Omega;\bbR^3)$ such that $P_n(\mu_n)$ is equal to $\mu_n$ on the lattice nodes and linear on each element of the partition. 
\end{Def}

\begin{Def}
\label{def_converge}
We say that $\bmmu\in W$ converges to $m\in H^1(\Omega;\bbR^3)$ if $P_n(\mu_n)$ converges weakly in $H^1(\Omega;\bbR^3)$ to $m$ as $n\to\infty$. 
We denote this $\mu_n \xrightarrow[n\to\infty]{}m$.
\end{Def}

\begin{Def}
We define the projection $p_n: \cC(\Omega;\bbR^3) \longrightarrow W_n$, such that for all $m\in\cC(\Omega;\bbR^3)$, $p_n(m)\in W_n$ and for all $x\in\cL_{n,\Omega}$, $(p_n(m))_x = m(x)$.
\end{Def}

\subsection{Main result}

The exchange energy $\cE_{n,\rmex}$ is a functional defined on $W_n$. 
The main result of this paper is the following.

\begin{Th}
\label{th_main}
Let $(\mu_n)_{n\in\bbN^*}\in\gH$. In the sense of the topology defined by Definition \ref{def_converge},
\begin{equation*}
\cE_{n,\rmex} \xrightarrow[n\to\infty]{\Gamma} \cE_{\infty,\rmex},
\end{equation*}
where
\begin{equation*}
\cE_{\infty,\rmex}: \begin{array}[t]{rcl}
H^1(\Omega;\bbR^3) & \longrightarrow & \bbR_+ \\
m & \longmapsto & \displaystyle 2 A \int_{\Omega}|\nabla m(x)|^2\d x.
\end{array}
\end{equation*}
Moreover, this convergence is compatible with the unit norm constraint.
\end{Th}

\begin{Rmk}
The $\Gamma$-convergence result is two-fold \cite{Braides02}: 
\begin{description}
\item[construction:] for all $m\in H^1(\Omega;\bbR^3)$, there exists $\bmmu\in W$ such that $\mu_n\to m$ and $\limsup_{n\to\infty} \cE_{n,\rmex}(\mu_n) \leq  \cE_{\infty,\rmex}(m)$.
\item[lower semi-continuity:] for all $\bmmu\in W$, such that $\|\mu_n\| \leq B$, $\mu_n\to m_\infty\in H^1(\Omega;\bbR^3)$ and $\liminf_{n\to\infty} \cE_{n,\rmex}(\mu_n) \geq  \cE_{\infty,\rmex}(m_\infty)$.
\end{description}
\end{Rmk}

\begin{proof}
The proof of Theorem \ref{th_main} splits into many steps to which various lemmas are devoted, the technical proofs of which are postponed.
 
\begin{Lem}
\label{lem_converge}
Let $(\mu_n)_{n\in\bbN^*}\in\gH$. Then there exists $m_\infty\in H^1(\Omega;\bbR^3)$ such that $\mu_n \xrightarrow[n\to\infty]{}m_\infty$.
\end{Lem}

Now this function $m_\infty$ is only in $H^1(\Omega;\bbR^3)$, and not continuous, and we need to use pointwise values of this function.
Therefore, for all $\eps$, we approximate $m_\infty$ by a $\cC^1(\Omega;\bbR^3)$ which coincides with $m_\infty$ on a smaller domain $\Omega_\eps$. We denote by $\lambda$ the Lebesgue measure on $\bbR^d$.

\begin{Lem}
\label{lem_approx}
Let $u \in H^1(\Omega;\bbR^3)$, for all $\eps>0$, there exists an open set $\omega_{u,\eps}$ and a function $u_\eps \in  \cC^1(\Omega;\bbR^3)$ such that for all $x\in  \Omega\setminus\omega_{u,\eps},\ u_\eps(x)=u(x)$, $\lambda(\omega_{u,\eps})<\eps$, and 
\begin{equation*}
\lim_{n\to\infty} \frac{\#(\cL_{n,\Omega}\cap \omega_{u,\eps})}{\#(\cL_{n,\Omega})} < \eps.
\end{equation*}
\end{Lem}

\begin{Cor}
Let $u \in H^1(\Omega,\bbR^3)$; for all $\eps>0$, there exists an open set $\Omega_{u,\eps}\subset\Omega$ with piecewise $\cC^1$ boundary such that $\lambda(\Omega\setminus\Omega_{u,\eps})\leq\eps$ and
\begin{equation*}
u_{|\Omega_{u,\eps}} \in \cC^1(\Omega_{u,\eps};\bbR^3).
\end{equation*}
\end{Cor}

\begin{Lem}
\label{lem_construct}
Let $u\in \cC^1(\Omega;\bbR^3)$. Then $(p_n(u))_{n\in\bbN^*}\in\gH$, $p_n(u) \xrightarrow[n\to\infty]{}u$ and $\lim_{n\to\infty} \cE_{n,\rmex}(p_n(u)) = \cE_{\infty,\rmex}(u)$.
Moreover if $u$ has a constant unit norm, then $(p_n(u))_{n\in\bbN^*}\in\gH_1$.
\end{Lem}

For all $\eps>0$, we can therefore associate to $m_\infty$ an open set $\Omega_\eps$ on which it is $\cC^1$ and define the approximate energies:
\begin{eqnarray*}
\cE^\eps_{n,\rmex}(\mu_n) 
& = & \frac{a}{n} \sum_{x\in\cL_{n,\Omega_\eps}} \sum_{y\in N_{n,\Omega_\eps,x}} A |\mu_{n,y}-\mu_{n,x}|^2\hspace{1cm} \textrm{ for all }n\in\bbN^*, \\
\cE^\eps_{\infty,\rmex}(m_\infty) 
& = & 2 A \int_{\Omega_\eps} |\nabla m_\infty(x)|^2\d x.
\end{eqnarray*}
Comparing $\cE^\eps_{n,\rmex}(\mu_n)$ and $\cE^\eps_{n,\rmex}(p_n(m_\infty))$, we first show the following.

\begin{Lem}
\label{lem_sci}
For all $\eps>0$, $\liminf_{n\to\infty} \cE^\eps_{n,\rmex}(\mu_n) \geq  \cE^\eps_{\infty,\rmex}(m_\infty)$.
\end{Lem}
Estimating the remainder of the sums and integrals on $\omega_\eps$, we can then prove the following.
\begin{Lem}
\label{lem_concl}
$\liminf_{n\to\infty} \cE_{n,\rmex}(\mu_n) \geq  \cE_{\infty,\rmex}(m_\infty)$.
\end{Lem}
\end{proof}

\subsubsection{Proof of Lemma \ref{lem_converge}: Limit of a lattice of spins}
\label{sec_elements}

Lemma \ref{lem_converge} is proved using an explicit computation of the projection $P_n(\mu_n)$. 
To this aim, we define, for all $n\in\bbN^*$,
\begin{itemize}
\item $\cT_n$, the set of tetrahedra which form the partition of $\cL_{n,\Omega}$;
\item $T_n$, the set of corner tetrahedra (4 for each mesh of the lattice, see Figure \ref{fig_tetra});
\item $T^*_n$, the set of center tetrahedra (1 for each mesh of the lattice, see Figure \ref{fig_tetra});
\item $E_n$, the set of edges of the mesh of $ \cL_{n,\Omega}$ (given by couples of the indices of the lattice nodes); 
\item $C_n$, the set of edges of elements of $T^*_n$;
\item $S_n$, the set of outer surfaces of $T_n$ (triplets $(i,j,k)$ where $(i,j)\in C_n$, $(i,k)$ and $(j,k) \in E_n$).
\end{itemize}

\begin{figure}[H]
\begin{center}
\begin{tikzpicture}[scale=2]
\draw(0,0) -- (1,0);
\draw(0,0) -- (0,1);
\draw(0,1) -- (1,1);
\draw(1,0) -- (1,1);
\draw(0,1) -- (0,1,-1);
\draw(0,1,-1) -- (1,1,-1);
\draw[very thick](1,1) -- (1,1,-1);
\draw(1,0) -- (1,0,-1);
\draw(1,0,-1) -- (1,1,-1);
\draw[dashed](0,0) -- (0,0,-1);
\draw[dashed](0,0,-1) -- (1,0,-1);
\draw[dashed](0,0,-1) -- (0,1,-1);
\draw[very thick](1,1) -- (1,0,-1);
\shadedraw[shading=axis,shading angle=-45] (0,1) -- (1,1) -- (0,1,-1);
\shadedraw[shading=axis,shading angle=180] (0,0) -- (1,0) -- (0,1);
\draw[->] (-.5,1) node[left] {$s\in S_n$} -- (.3,1,-.3); 
\draw[->] (-.5,.9) -- (.3,.3,0); 
\draw[->] (1.5,1,-.7) node[right] {$e\in E_n$} -- (1,1,-.7);
\draw[->] (1.5,.5,-.5) node[right] {$e\in C_n$} -- (1,.5,-.5);
\end{tikzpicture}
\end{center}
\caption{\label{fig_sets}Different sets of tetraedra and edges.}
\end{figure}
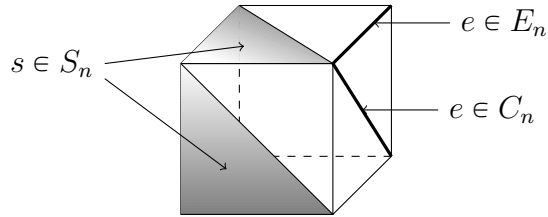

\begin{Rmk}
\label{rmk_polyhedra}
Generically $\Omega_n\equiv\bigcup_{\tau\in\cT_n} \tau \varsubsetneq \Omega$.
We suppose that the geometry of $\Omega$ is such that $\lambda(\Omega\setminus \Omega_n) = O(1/n)$, i.e. there are order $n^2$ tetrahedra covering the difference set. 
We set $P_n(\mu_n)$ to zero on $\Omega\setminus \Omega_n$ . 
\end{Rmk}

\noindent
\textbf{Step 1. Estimate of $\int_\Omega|P_n(\mu_n)(x)|^2\d x$.}

Let $x\in\Omega_n$, then there exist $\tau\in\cT_n$ and $x\in\tau$.
Since $P_n(\mu_n)$ is linear on $\tau$, then $P_n(\mu_n)(x)$ is a pondered  mean of the $\mu_{n,x_i^\tau}$, $i=1,\dots,4$, where $(x_1^\tau,x_2^\tau,x_3^\tau,x_4^\tau)\in(\cL_{n,\Omega})^4$ are the vertices of $\tau$.
Therefore $|P_n(\mu_n)(x)|\leq1$. 
Hence
\begin{eqnarray*}
\int_\Omega |P_n(\mu_n)(x)|^2\d x 
& = & \int_{\Omega_n} |P_n(\mu_n)(x)|^2\d x 
= \sum_{\tau\in\cT_n} \int_\tau |P_n(\mu_n)(x)|^2\d x \\
& \leq & \sum_{\tau\in\cT_n} \lambda(\tau) = \lambda(\Omega_n).
\end{eqnarray*}
 
\noindent
\textbf{Step 2. Explicit computation of $\int_\Omega|\nabla P_n(\mu_n)(x)|^2\d x$.}

We compute $\int_{\Omega}|\nabla P_n(\mu_n)(x)|^2\d x$ explicitly using the tetrahedron decomposition of the lattice.
On each tetrahedron $P_n(\mu_n)$ is linear, and therefore its gradient is constant. 
 
If $\tau\in T_n$, we can construct an orthogonal system using the lattice nodes $(x_1^\tau,x_2^\tau,x_3^\tau,x_4^\tau)\in(\cL_{n,\Omega})^4$, for example $(\mu_{n,x_i^\tau}-\mu_{n,x_1^\tau})_{i=2,\dots,4}$.
Each direction yields one component of the gradient.
Since the length of the edges are $a/n$, the components of the gradient are   the $(\mu_{n,x_i^\tau}-\mu_{n,x_1^\tau})n/a$. 
Besides the volume of $\tau$ is $\frac 16 (a/n)^3$, therefore 
\begin{equation*}
\int_\tau |\nabla P_n(\mu_n)(x)|^2 \d x 
= \frac{a}{6n} \sum_{i=2}^4 |\mu_{n,x_i^\tau}-\mu_{n,x_1^\tau}|^2.
\end{equation*}
For a center tetrahedron $\tau*\in T^*_n$, it is a bit more tricky since the edges are not orthogonal.
The volume is of course the complement of the volumes of the corner tetrahedra, namely $\frac13 (a/n)^3$.
The computation yields 
\begin{equation*}
|\nabla P_n(\mu_n)(x)|^2 = \frac{n^2}{4a^2} \sum_{1\leq i,j \leq4} |\mu_{n,x_i^{\tau^*}}-\mu_{n,x_j^{\tau^*}}|^2
\end{equation*}
and hence
\begin{equation*}
\int_{\tau*} |\nabla P_n(\mu_n)(x)|^2 \d x 
= \frac{a}{12n} \sum_{1\leq i,j \leq4} |\mu_{n,x_i^{\tau^*}}-\mu_{n,x_j^{\tau^*}}|^2.
\end{equation*}
 
Gathering all the contributions
\begin{equation*}
\begin{aligned}
\int_\Omega |\nabla P_n&(\mu_n)(x)|^2 \d x \\
& = \sum_{\tau \in \cT_n} \int_\tau |\nabla P_n(\mu_n)(x)|^2 \d x \\
& = \sum_{\tau \in T_n} \int_\tau |\nabla P_n(\mu_n)(x)|^2 \d x 
+ \sum_{\tau* \in T^*_n} \int_{\tau*} |\nabla P_n(\mu_n)(x)|^2 \d x \\ 
& = 4 \sum_{(i,j)\in E_n} \frac{a}{6n} |\mu_{n,x_i}-\mu_{n,x_j}|^2 
+ 2 \sum_{(i,j)\in C_n} \frac{a}{12n} |\mu_{n,x_i}-\mu_{n,x_j}|^2 - \cS_{n,\Omega}.
\end{aligned}
\end{equation*} 
The coefficient 4 in front of the sum on $E_n$ occurs because each element of $E_n$ is an element of four $\tau\in T_n$.
Similarly each element of $C_n$ belongs to two $\tau^*\in T^*_n$. 
The positive error $\cS_{n,\Omega}$ is due to an over-estimation because some of the edges $e\in E_n$ are on the outer surface of $\Omega_n$ and have been counted too many times. 
Following Remark \ref{rmk_polyhedra}, the contribution of $\cS_{n,\Omega}$ will always be $O(1/n)$ less than that of the other terms, and therefore will tend to zero as $n\to\infty$.

We rewrite the first sum
\begin{equation*}
4 \sum_{(i,j)\in E_n} \frac{a}{6n} |\mu_{n,x_i}-\mu_{n,x_j}|^2 =  \sum_{x\in\cL_{n,\Omega}} \sum_{y\in N_{n,\Omega,x}} \frac{a}{3n} |\mu_{n,x}-\mu_{n,y}|^2
\end{equation*}
(here each edge is counted twice through the couples $(x,y)$ and $(y,x)$).
For $x \in l_{n,\Omega}$, we can only estimate $(a/3n) |\mu_{n,x}-\mu_{n,y}|^2 \leq 4a/3n$, but there are only $O(n)$ such terms.
For $x\in\cL_{n,\Omega}\setminus l_{n,\Omega}$, $(a/3n) |\mu_{n,x}-\mu_{n,y}|^2 \leq ac/3n^3$ and there are $O(n^3)$ such terms.
Therefore
\begin{equation*}
4 \sum_{(i,j)\in E_n} \frac{a}{6n} |\mu_{n,x_i}-\mu_{n,x_j}|^2 = O(1).
\end{equation*}

For the second sum
\begin{equation*}
\begin{aligned}
& \sum_{(i,j)\in C_n} \frac{a}{6n} | \mu_{n,x_i}-\mu_{n,x_j}|^2 
= \frac12\sum_{(i,j,k)\in S_n} \frac{a}{6n} |\mu_{n,x_i}-\mu_{n,x_k} + \mu_{n,x_k}-\mu_{n,x_j}|^2 \\
& \begin{aligned}
= \sum_{(i,j,k)\in S_n} \frac{a}{12n} \Big( |\mu_{n,x_i}-\mu_{n,x_k}|^2 & + |\mu_{n,x_j}-\mu_{n,x_k}|^2 \\
  & + 2(\mu_{n,x_i}-\mu_{n,x_k})\cdot(\mu_{n,x_k}-\mu_{n,x_j})\Big) 
  \end{aligned} \\
& = \sum_{(i,j)\in E_n} \frac{a}{3n} |\mu_{n,x_i}-\mu_{n,x_j}|^2 +  \sum_{(i,j,k)\in S_n}   \frac{a}{6n} (\mu_{n,x_i}-\mu_{n,x_k})\cdot(\mu_{n,x_k}-\mu_{n,x_j}) \\
& = \sum_{x\in\cL_{n,\Omega}} \sum_{y\in N_{n,\Omega,x}} \frac{a}{6n} |\mu_{n,x}-\mu_{n,y}|^2 +  \sum_{(i,j,k)\in S_n} \frac{a}{6n} (\mu_{n,x_i}-\mu_{n,x_k})\cdot(\mu_{n,x_k}-\mu_{n,x_j}). 
\end{aligned}
\end{equation*}
As for the previous sum, we can decompose these sums into two $O(1)$ contributions. 
There are at most $O(n)$ terms contributing to $\cS_{n,\Omega}$ and stemming from an $x\in l_{n,\Omega}$, therefore $\cS_{n,\Omega}=O(1)$.
Therefore $\|P_n(\mu_n)\|_{H^1(\Omega;\bbR^3)}$ is uniformly bounded and $P_n(\mu_n)$ is weakly convergent in $H^1(\Omega;\bbR^3)$.
This, together with Definition \ref{def_converge}, leads to Lemma \ref{lem_converge}. 

A by-product of this proof is the fact that we can write
\begin{equation*}
\|\nabla P_n(\mu_n)\|^2_{L^2(\Omega;\cL(\bbR^3;\bbR^3))} 
= \int_\Omega |\nabla P_n(\mu_n)(x)|^2 \d x 
= \frac1{2A} \cE_{n,\rmex}(\mu_n) + \alpha_n(\mu_n)
\end{equation*}
where
\begin{equation*}
\alpha_n(\mu_n) = \frac{a}{6n} \sum_{(i,j,k)\in S_n} (\mu_{n,x_i}-\mu_{n,x_k})\cdot(\mu_{n,x_k}-\mu_{n,x_j}) - \cS_{n,\Omega}.
\end{equation*}

\subsubsection{Proof of Lemma \ref{lem_approx}: $\cC^1$ approximation of a $H^1$ function}

Following Ziemer theorem (see \cite{Ziemer89}, Theorem 3.11.6), we know that for any function $u\in H^1(\Omega;\bbR^3)$ and for all $\eps>0$, there exists a function $u_\eps\in\cC^1(\Omega;\bbR^3)$ such that $\lambda(\omega_{u,\eps})\leq\eps$, where 
\begin{equation*}
\omega_{u,\eps} := \{x\in\Omega \textrm{ such that } u(x) \neq u_\eps(x)\}.
\end{equation*}

We want to extend this result and be able to localize the irregularities of $u$ with respect to a shrinking lattice.

Let $u\in H^1(\Omega;\bbR^3)$, and $X_u$ be the set of points where $u$ is not $\cC^1$.
Since the gradient of $u\in H^1(\Omega;\bbR^3)$ has to be defined almost everywhere, there cannot be an open ball in $X_u$ and therefore $\overset{\circ}{\overline{X_u}}=\emptyset$.

We now fix $\eps>0$. 
The Lebesgue and $\cH^3$ Hausdorff measures coincide in $\bbR^3$ and therefore we both have $\lambda(\overline{X_u})\leq\eps$ and $\cH^3(\overline{X_u})\leq\eps$.\\
Hence there exists a sequence of open balls $(B_i)_{i\in\bbN}$ such that 
\begin{equation*}
\overline{X_u}\subset \bigcup_{i\in\bbN} B_i \text{ and } \sum_{i=1}^\infty \diam(B_i)^3 < \eps.
\end{equation*}
 
Since $\overline{X_u}$ is closed and bounded, it is compact and we can extract from this open cover a finite subcover.

Hence there exists $N\in\bbN$ such that
\begin{equation*}
\overline{X_u}\subset \bigcup_{i=1}^N B_i \text{ and } \sum_{i=1}^N\diam(B_i)^3 < \eps.
\end{equation*}
 
Obviously $\bigcup_{i=0}^N B_i$ is bounded with a piecewise $\cC^1$ boundary, hence 
\begin{equation*}
\lim_{n\to\infty} \frac{\#(\cL_{n,\Omega}\bigcap \bigcup_{i=1}^N B_i )}{\#(\cL_{n,\Omega})} = \lambda \left(\bigcup_{i=1}^N B_i \bigcap \Omega\right)< \eps.
\end{equation*}

Besides $u_{|\bigcup_{i=1}^N B_i } \in H^1(\bigcup_{i=1}^N B_i;\bbR^3)$, and we therefore can choose $u_\eps$ such that $\omega_{u,\eps} \subset \bigcup_{i=1}^N B_i$.
Finally
\begin{equation*}
\lim_{n\to\infty} \frac{\#(\cL_{n,\Omega}\bigcap \omega_{u,\eps})}{\#(\cL_{n,\Omega})} \leq \eps.
\end{equation*}

In our proof, we set $\Omega_\eps=\widering{\Omega\setminus\omega_{m_\infty,\eps}}$ and begin to work on the restricted shrunk lattice $\cL_{n,\Omega_\eps} = \cL_{n,\Omega} \bigcap \Omega_{\eps}$.
We also denote $\cD_{n,\Omega_\eps}$ the subset of elements $x\in \cL_{n,\Omega_\eps}$ such that $\#(N_{n,\Omega,x}\bigcap\Omega_\eps) \neq 6$, that is the set of nodes which are too close to $\partial\Omega_\eps$ to have their 6 nearest neighbors in $\Omega_\eps$.\\

Since $\partial\Omega_\eps$ is piecewise $\cC^1$ for all $\eps>0$, we know that $\# \cD_{n,\Omega_\eps} = O(n^2)$.

\subsubsection{Proof of Lemma \ref{lem_construct}: Construction} 

Let $x\in\cL_{n,\Omega}$ and $y\in N_{n,\Omega,x}$.
Since $u\in\cC^1(\Omega;\bbR^3)$, $\nabla u$ is bounded by some constant $C$ on $\Omega$ and 
\begin{equation*}
|p_n(u)_x-p_n(u)_y|^2 \leq C^2|x-y|^2 \leq C^2 \left(\frac an\right)^2 
= \frac{C^2a^2}{n^2}.
\end{equation*}
Therefore $p_n(u)\in\gH$.
Clearly if $|u|=1$ on $\Omega$, for all $x\in\cL_{n,\Omega}$, $|p_n(u)_x|=1$ and $p_n(u)\in\gH_1$.

Lemma \ref{lem_converge} implies that there exists $u_\infty\in H^1(\Omega;\bbR^3)$ such that $p_n(u) \xrightarrow[n\to\infty]{} u_\infty$.
Since $u\in\cC^1(\Omega;\bbR^3)$, $P_n(p_n(u))$ converges towards $u$. 
This convergence is pointwise and even uniform on $\Omega$.
Hence $u_\infty=u$.

Last
\begin{eqnarray*}
\cE_{n,\rmex}(p_n(u)) 
& = & \frac{a}{n} \sum_{x\in\cL_{n,\Omega}} \sum_{y\in N_{n,\Omega,x}} A |p_n(u)_y-p_n(u)_x|^2 \\
& = & \frac{a}{n} \sum_{x\in\cL_{n,\Omega}} \sum_{y\in N_{n,\Omega,x}} A |u(y)-u(x)|^2 .\\
\end{eqnarray*}
Since $u\in\cC^1(\Omega;\bbR^3)$, for all $x\in\cL_{n,\Omega}$ and all $y\in N_{n,\Omega,x}$, 
\begin{equation*}
\frac{|u(y)-u(x)|^2}{|y-x|^2} = \left|\nabla u(x)\cdot\frac{y-x}{|y-x|}\right|^2 + e_{n,x,y},
\end{equation*}
and therefore for all $x\in\cL_{n,\Omega}\setminus\cD_{n,\Omega}$
\begin{equation*}
\left(\frac{n}{a} \right)^2 \sum_{y\in N_{n,\Omega,x}} |u(y)-u(x)|^2 = 2|\nabla u(x)|^2 + e_{n,x},
\end{equation*} 
and the errors $e_{n,x}$ are uniformly $o(1)$ as $n\to\infty$.
Hence
\begin{equation*}
\cE_{n,\rmex}(p_n(u)) = 2A \sum_{x\in\cL_{n,\Omega}} \left(\frac{a}{n}\right)^3 (|\nabla u(x)|^2 + e_{n,x}) + O\left(\frac1n\right),
\end{equation*}
where the $O(1/n)$ stems from $x\in\cD_{n,\Omega}$.
Since $\lim_{n\to\infty} \sum_{x\in\cL_{n,\Omega}} (a/n)^3 e_{n,x} = 0$, and the border of $\Omega$ is piecewise $\cC^1$
\begin{equation*}
\lim_{n\to\infty} \sum_{x\in\cL_{n,\Omega}} \left(\frac{a}{n}\right)^3 |\nabla u(x)|^2 = \int_{\Omega} |\nabla u(x)|^2 \d x
\end{equation*}
and 
\begin{equation*}
\lim_{n\to\infty} \cE_{n,\rmex}(p_n(u)) = \cE_{\infty,\rmex}(u).
\end{equation*}

\subsubsection{Proof of Lemma \ref{lem_sci}: Lower semi-continuity}

With Definition \ref{def_converge} for the convergence, $P_n(\mu_n) \rightharpoonup m_\infty$ in $H^1(\Omega;\bbR^3)$, and 
\begin{equation*}
\liminf_{n\to\infty} \|P_n(\mu_n)\|_{H^1(\Omega;\bbR^3)} \geq \|m_\infty\|_{H^1(\Omega;\bbR^3)}
\end{equation*}
(see \cite{Braides02}, Proposition 2.3).

Since we have assumed that $\Omega$ is compact and has a piecewise $\cC^1$ boundary, we have already seen that the convergence is strong in $L^2(\Omega;\bbR^3)$ and therefore 
\begin{equation*}
\liminf_{n\to\infty} \| \nabla P_n(\mu_n)\|^2_{L^2(\Omega;\cL(\bbR^3; \bbR^d))} \geq \|\nabla m_\infty\|^2_{L^2(\Omega;\cL(\bbR^3;\bbR^d))}.
\end{equation*}

Let us first fix $\eps$ and work in $\Omega_\eps$.
Thanks to Proposition \ref{lem_construct}, we know that $m_\infty$ is continuous and 
\begin{equation*}
\|\nabla P_n(p_n(m_\infty))\|_{L^2(\Omega_\eps;\cL(\bbR^3;\bbR^d))} \xrightarrow[n\to\infty]{}  \| \nabla m_\infty\|_{L^2(\Omega_\eps;\cL(\bbR^3;\bbR^d))}.
\end{equation*}
This implies that
\begin{equation*}
\liminf_{n\to\infty} \left(\|\nabla P_n(\mu_n)\|^2_{L^2(\Omega_\eps;\cL(\bbR^3;\bbR^d))}
- \|\nabla P_n(p_n(m_\infty))\|^2_{L^2(\Omega_\eps;\cL(\bbR^3;\bbR^d))} \right) \geq 0.
\end{equation*}
According to Lemma \ref{lem_converge}, we can write
\begin{equation*}
\|\nabla P_n(\mu_n)\|^2_{L^2(\Omega_\eps;\cL(\bbR^3;\bbR^d))} 
= \frac1{2A}( \cE^\eps_{n,\rmex}(\mu_n) +\alpha_n^\eps(\mu_n)).
\end{equation*}
We therefore know that
\begin{equation*}
\liminf_{n\to\infty}\left(\frac1{2A} \cE^\eps_{n,\rmex}(\mu_n) - \frac1{2A} \cE^\eps_{n,\rmex}(p_n(m_\infty)) + \alpha_n^\eps(\mu_n)- \alpha_n^\eps(p_n(m_\infty)) \right) \geq 0.
\end{equation*}
We split this into 5 parts, namely
\begin{eqnarray*}
\frac1{2A} \cE^\eps_{n,\rmex}(\mu_n) - \frac1{2A}  \cE^\eps_{n,\rmex}(p_n(m_\infty)) & = & a_n + b_n, \\
\alpha_n^\eps(\mu_n)- \alpha_n^\eps(p_n(m_\infty)) & = & c_n+d_n+e_n +f_n,
\end{eqnarray*}
which can be expressed using $\psi_{n,x}:= \mu_{n,x}-m_\infty(x)$, defined for all $n\in \bbN^*$ and $x\in \cL_{n,\Omega_\eps}$:
\begin{eqnarray*}
a_n & = & \frac{a}{2n} \sum_{x\in\cL_{n,\Omega_\eps}} \sum_{y\in N_{n,\Omega_\eps,x}} |\psi_{n,x}-\psi_{n,y}|^2, \\
b_n & = & \frac{a}{2n} \sum_{x\in\cL_{n,\Omega_\eps}} \sum_{y\in N_{n,\Omega_\eps,x}} 2(m_\infty(x)-m_\infty(y))\cdot(\psi_{n,x}-\psi_{n,y}), \\
c_n & = & \frac{a}{6n} \sum_{(i,j,k)\in S_n} (m_\infty(x_i)-m_\infty(x_k))\cdot(\psi_{n,x_k}-\psi_{n,x_j}), \\
d_n & = & \frac{a}{6n} \sum_{(i,j,k)\in S_n} (\psi_{n,x_i}-\psi_{n,x_k})\cdot(m_\infty(x_k)-m_\infty(x_j)), \\
e_n & = & \frac{a}{6n} \sum_{(i,j,k)\in S_n} (\psi_{n,x_i}-\psi_{n,x_k})\cdot(\psi_{n,x_k}-\psi_{n,x_j}),\\
f_n  & = & \cS_{n,\Omega_\eps}(p_n(m_\infty)) - \cS_{n,\Omega_\eps}(\mu_n).
\end{eqnarray*}
We show below that $e_n\leq a_n$, and $b_n$, $c_n$, $d_n$ and $f_n$ tend to zero.
This implies that
\begin{equation*}
\liminf_{n\to\infty} (\alpha_n(\mu_n)-\alpha_n(p_n(m_\infty))) 
\leq \frac1{2A} \liminf_{n\to\infty} (\cE^{\eps}_{n,\rmex}(\mu_n)-\cE^{\eps}_{n,\rmex}(p_n(m_\infty)))
\end{equation*}
and therefore 
\begin{equation*}
\liminf_{n\to\infty} (\cE^{\eps}_{n,\rmex}(\mu_n) - \cE^{\eps}_{n,\rmex}(p_n(m_\infty))) \geq 0.
\end{equation*}
Lemma \ref{lem_construct} implies that $\lim_{n\to\infty} \cE^{\eps}_{n,\rmex}(p_n(m_\infty)) = \cE^{\eps}_{\infty,\rmex}(m_\infty)$, hence
\begin{equation*}
\liminf_{n\to\infty} \cE^{\eps}_{n,\rmex}(\mu_n) \geq \cE^{\eps}_{\infty,\rmex}(m_\infty).
\end{equation*}
which ends the proof.

\paragraph{Proof of $e_n\leq a_n$.}

For all $x_i$, $x_j$, $x_k\in\cL_{n,\Omega_\eps}$, 
\begin{equation*}
(\psi_{n,x_i}-\psi_{n,x_k})\cdot(\psi_{n,x_k}-\psi_{n,x_j}) 
\leq \frac12 (|\psi_{n,x_i}-\psi_{n,x_k}|^2+|\psi_{n,x_k}-\psi_{n,x_j}|^2).
\end{equation*}
Since each couple $(i,j)\in E_n$ is an element of 4 triples in $S_n$, we have
\begin{eqnarray*}
\sum_{(i,j,k)\in S_n} (\psi_{n,x_i}-\psi_{n,x_k})\cdot(\psi_{n,x_k}-\psi_{n,x_j}) 
& \leq & 2 \sum_{(i,j)\in E_n}|\psi_{n,x_i}-\psi_{n,x_j}|^2 \\
& \leq & \sum_{x\in\cL_{n,\Omega_\eps}} \sum_{y\in N_{n,\Omega_\eps,x}} |\psi_{n,x}-\psi_{n,y}|^2,
\end{eqnarray*}
which is a much stronger result than $e_n\leq a_n$. 
 
\paragraph{Proof of $c_n$ and $d_n\to0$.}
 
Let $(\psi_{n,x}-\psi_{n,y})\cdot(m_\infty(y)-m_\infty(z))$ be one term of the sum in $c_n$ and set $v=z-y$.
Then 
\begin{equation*}
(\psi_{n,x}-\psi_{n,y})\cdot(m_\infty(y)-m_\infty(z))
= (\psi_{n,x}-\psi_{n,y})\cdot(m_\infty(y)-m_\infty(y+v)) 
\end{equation*}
and in the same sum there is also a term $(\psi_{n,x}-\psi_{n,y})\cdot(m_\infty(y)-m_\infty(y-v))$, except for $y\in \cD_{n,\Omega_\eps}$.

Now since $\mu_n\in\gH$, there exists a constant $C_\psi\geq0$ such that 
\begin{equation*}
|\psi_{n,x}-\psi_{n,y}| \leq C_\psi \frac an,
\end{equation*}
except for $y \in l_{n,\Omega_\eps}$, but since there are $O(n)$ such nodes, their contribution in $c_n$ tends to 0.
We also have
\begin{equation*}
\begin{aligned}
\frac{n}a((m_\infty(y)-m_\infty(y-v)) & + (m_\infty(y)-m_\infty(y+v))) \\
& \xrightarrow[n\to\infty]{} (\nabla m_\infty(y)-\nabla m_\infty(y))\cdot \frac{v}{|v|} = 0,
\end{aligned}
\end{equation*}
and therefore
\begin{eqnarray*}
\frac{a}{6n} \sum_{(i,j,k)\in S_n}(\psi_{n,x_i}-\psi_{n,x_k})\cdot(m_\infty(x_k)-m_\infty(x_j)) 
& = & O\left(\frac1n\right)O(n^3)O\left(\frac1n\right)o\left(\frac1n\right) \\
& = & o(1).
\end{eqnarray*}

\begin{Rmk}
When $y\in\cD_{n,\Omega_\eps}$, we can only say that $|\psi_{n,x}-\psi_{n,y}| \leq C_\psi a/n$ (except on $l_{n,\Omega_\eps}$) and $|m_\infty(y)-m_\infty(z)| \leq C_m a/n$, and since $\#\cD_{n,\Omega_\eps}=O(n^2) $, the contribution of these nodes in $c_n$ is a 
\begin{equation*}
O(n^2)O\left(\frac1n\right)O\left(\frac1n\right)O\left(\frac1n\right)+O(n)O\left(\frac1n\right)O\left(\frac1n\right) = O\left(\frac1n\right).
\end{equation*}
\end{Rmk}

The sum $d_n$ is treated in the same way.

\paragraph{Proof of $f_n\to0$.}

The quantity $\cS_{n,\Omega_\eps}(\mu_n)$ is a sum of $O(n^2)$ terms reading like $(a/6n) |\mu_{n,x_i}-\mu_{n,x_j}|^2 $ for $(i,j)\in E_n$, or $(a/12n) |\mu_{n,x_i}-\mu_{n,x_j}|^2$ for $(i,j)\in C_n$. 
By Hypothesis \ref{hyp_local2}, only $O(n)$ among these terms can be only $o(1/n)$ and the others are $O(1/n^3)$. 
Hence $\cS_{n,\Omega_\eps}(\mu_n) = o(1)$.\\
On the other hand, the fact that $u_\infty \in \cC^1(\Omega_\eps, \bbR^3)$ ensures that all the terms in $\cS_{n,\Omega_\eps}(p_n(m_\infty))$ are $O(1/n^3)$, and therefore $\cS_{n,\Omega_\eps}(p_n(m_\infty))=O(1/n)$.

\paragraph{Proof of $b_n\to0$.}

We use the fact that $\cC^\infty(\Omega_\eps;\bbR^3)$ is dense in $H^1(\Omega_\eps;\bbR^3)$ for the $\|\cdot\|_{H^1(\Omega_\eps;\bbR^3)}$ norm. 
Let $\eta>0$, there exists $m_\eta\in\cC^\infty(\Omega_\eps;\bbR^3)$ such that $\|m_\infty-m_\eta\|_{H^1(\Omega_\eps;\bbR^3)} \leq \eta$.
Then $b_n$ can be split into two contributions $b_n=b_n^\eta+b_n^{'\eta}$ where
\begin{eqnarray*}
b_n^\eta & = & \frac{a}{n} \sum_{x\in\cL_{n,\Omega_\eps}} \sum_{y\in N_{n,\Omega_\eps,x}} (m_\eta(x)-m_\eta(y))\cdot(\psi_{n,x}-\psi_{n,y}), \\
b_n^{'\eta} & = & \frac{a}{n} \sum_{x\in \cL_{n,\Omega_\eps}} \sum_{y\in N_{n,\Omega_\eps,x}} ((m_\infty-m_\eta)(x)-(m_\infty-m_\eta)(y))\cdot(\psi_{n,x}-\psi_{n,y}).
\end{eqnarray*}

For the first term, we notice that
\begin{eqnarray*}
b_n^\eta
& = & \frac{a}{n} \sum_{x\in\cL_{n,\Omega_\eps}} \sum_{y\in N_{n,\Omega,x}} (m_\eta(x)-m_\eta(y)) \cdot \psi_{n,x} \\
& - & \frac{a}{n} \sum_{y\in\cL_{n,\Omega_\eps}} \sum_{x\in N_{n,\Omega,y}} (m_\eta(x)-m_\eta(y)) \cdot \psi_{n,y}.
\end{eqnarray*}
As in the previous proof, we write $y$ as $x+v$ and
\begin{equation*}
b_n^\eta 
= \frac{a}{n} \sum_{x\in\cL_{n,\Omega_\eps}} \sum_{v\in V_{n,0}} 
( (m_\eta(x)-m_\eta(x+v))-(m_\eta(x-v)-m_\eta(x)) ) \cdot \psi_{n,x}.
\end{equation*}
We estimate
\begin{equation*}
\begin{aligned}
(m_\eta(x)-m_\eta(x+v))&-(m_\eta(x-v)-m_\eta(x)) \\
&= \int_{t=1}^0 \nabla m_\eta (x+tv)\cdot v\ \d t 
- \int_{t=0}^{-1} \nabla m_\eta (x+tv)\cdot v\ \d t \\
& = - \nabla m_\eta (x)\cdot \frac{v}{|v|} 
+ \nabla m_\eta (x)\cdot \frac{v}{|v|} + O(|v|^{2}) \leq \frac{ca^2}{n^2}, 
\end{aligned}
\end{equation*}
where the constant $c$ only depends on the second derivative of $m_\eta$, which is bounded on $\Omega_\eps$. 
Thus
\begin{equation*}
|b_n^\eta| \leq \frac{ca^3}{n^3}  \sum_{x\in\cL_{n,\Omega_\eps}} \sum_{y\in N_{n,\Omega,x}} |\psi_{n,x}|.
\end{equation*}

Thanks to the compact injection of $H^1(\Omega;\bbR^3)$ in $L^1(\Omega; \bbR^3)$, we deduce that this term vanishes.
We obviously have
\begin{equation*}
|b_n^{'\eta}| \leq \frac{a^2C_\psi}{n^2} \sum_{x\in\cL_{n,\Omega_\eps}} \sum_{y\in N_{n,\Omega_\eps,x}} |(m_\infty-m_\eta)(x)-(m_\infty-m_\eta)(y)|.
\end{equation*}
For any function $u\in\cC^1(\Omega;\bbR^3)$, and $y=x+v$, if $[x,y]\subset \Omega$, then
\begin{eqnarray*}
u(y) - u(x) 
& = & \int_{t=1}^0 \nabla u(x+tv)\cdot v\ \d t \\
& = & \nabla u(x)\cdot v + \int_{t=1}^0 (\nabla u(x+tv)-\nabla u(x))\cdot v\ \d t.
\end{eqnarray*}
Applying this to $m_\infty$ and $m_\eta$,
\begin{eqnarray*}
(m_\infty-m_\eta)(y)-(m_\infty-m_\eta)(x)
& = & (\nabla m_\infty(x)-\nabla m_\eta(x)) \cdot v \\
& + & \int_{t=1}^0 (\nabla m_\infty(x+tv)- \nabla m_\infty(x)) \cdot v\ \d t \\
& - & \int_{t=1}^0 (\nabla m_\eta(x+tv)-\nabla m_\eta(x)) \cdot v\ \d t.
\end{eqnarray*}

\noindent
$\star$ By definition of $m_\eta$, $\|m_\infty-m_\eta\|_{H^1(\Omega;\bbR^3)} \leq \eta$, which implies that $\|\nabla m_\infty(x)-\nabla m_\eta(x)\|_{L^2(\Omega_\eps;\bbR^3)}\leq\eta$.
Hence
\begin{equation*}
\begin{aligned}
\frac{a^2C_\psi}{n^2} &\left| \sum_{x\in\cL_{n,\Omega_\eps}} \sum_{y\in N_{n,\Omega_\eps,x}}  (\nabla m_\infty(x)-\nabla m_\eta(x)).(y-x) \right| \\
& \leq 2 C_\psi \sum_{x\in\cL_{n,\Omega_\eps}} \left(\frac an\right)^3|\nabla m_\infty(x)-\nabla m_\eta(x)| \\
& \leq 2 C_\psi \left(\sum_{x\in\cL_{n,\Omega_\eps}} \left(\frac an\right)^3\right)^{1/2} \left(\sum_{x\in\cL_{n,\Omega_\eps}} \left(\frac an\right)^3|\nabla m_\infty(x)-\nabla m_\eta(x)|^2\right)^{1/2}. 
\end{aligned}
\end{equation*}
Since
\begin{equation*}
\lim_{n\to\infty} \sum_{x\in\cL_{n,\Omega_\eps} } \left(\frac{a}{n}\right)^3 |\nabla m_\infty(x)- \nabla m_\eta(x)|^2 = \int_{\Omega_\eps} |\nabla m_\infty(x)-\nabla m_\eta(x)|^2 \d x,
\end{equation*}
\begin{equation*}
\begin{aligned}
\lim_{n\to\infty} \frac{a^2C_\psi}{n^2} \sum_{x\in\cL_{n,\Omega_\eps}} & \sum_{y\in N_{n,\Omega_\eps,x}} (\nabla m_\infty(x)-\nabla m_\eta(x)).(y-x)| \\
& \leq C \|\nabla m_\infty(x)-\nabla m_\eta(x)\|_{L^2(\Omega_\eps;\bbR^3)} 
\leq C \eta.
\end{aligned}
\end{equation*}

\noindent
$\star$  Let us treat the second contribution involving the $\cC^\infty$ function $m_\eta$.
Since
\begin{equation*}
\int_{t=1}^0 (\nabla m_\eta(x+tv)-\nabla m_\eta(x))\cdot v\ \d t = \int_{t=1}^0 \int_{t'=0}^t \nabla^2 m_\eta (x+t'v)\cdot v\otimes v\ \d t' \d t,
\end{equation*}
and $\nabla^2 m_\eta$ is uniformly bounded on $\Omega_\eps$, we can estimate
\begin{equation*}
\frac{a^2C_\psi}{n^2} \sum_{x\in\cL_{n,\Omega_\eps} } \sum_{y\in N_{n,\Omega_\eps,x}} \int_{t=1}^0 (\nabla m_\eta(x+t(y-x))- \nabla m_\eta(x)) \cdot (y-x)\ \d t =O(\frac1n).
\end{equation*}

\noindent
$\star$ In the last contribution, function $m_\infty$ is only $\cC^1$:
\begin{equation*}
\begin{aligned}
\frac{a^2C_\psi}{n^2} & \sum_{x\in\cL_{n,\Omega_\eps} } \sum_{y\in N_{n,\Omega_\eps,x}} \left| \int_{t=1}^0 (\nabla m_\infty(x+t(y-x)) - \nabla m_\infty(x))\cdot (y-x)\ \d t \right| \\
& \leq  \frac{a^2C_\psi}{n^2} \sum_{x\in\cL_{n,\Omega_\eps} } \sum_{v\in N_{n,0}} \int_{t=0}^1 |(\nabla m_\infty(x+tv) - \nabla m_\infty(x))\cdot v|\ \d t \\
& \leq  \frac{a^3C_\psi}{n^3} \int_{t=0}^1 \sum_{x\in\cL_{n,\Omega_\eps} } \sum_{v\in N_{n,0}} | (\nabla m_\infty(x+tv)- \nabla m_\infty(x))\cdot\frac{v}{|v|}|\ \d t.
\end{aligned}
\end{equation*}
Since $m_\infty\in\cC^1(\Omega_\eps;\bbR^3)$, the elements of the sum converge uniformly towards 0 as $n\to\infty$, and therefore the integral is $o(n^{3})$. \\

Hence $\lim_{n\to\infty} |b_n^{'\eta}| \leq \eta$ and therefore $\lim_{n\to\infty} |b_n| \leq \eta$, for all $\eta>0$.
This leads to $\lim_{n\to\infty} |b_n|=0$, and the lemma is proved.

\subsubsection{Proof of Lemma \ref{lem_concl}}

The terms that occur in $\cE_{n,\rmex}(\mu_{n})$ but not in $\cE^\eps_{n,\rmex}(\mu_{n})$ are those involing couples $(x,y)$ where one at least of the nodes belong to $\Omega\setminus\Omega_\eps$. 
There are $\eps O(n^3)$ such nodes. 
Hence, by Hypothesis \ref{hyp_local2}, the contribution of $l_{n,\Omega}$ in $\cE_{n,\rmex}(\mu_{n})$ tends to zero, and there exists $c>0$ such that for all $\{\mu_n\}_{n\in\bbN^*} \in \gH$, 
\begin{equation*}
\lim_{n\to\infty} |\cE^\eps_{n,\rmex}(\mu_n)-\cE_{n,\rmex}(\mu_{n})|\leq c\eps.
\end{equation*}

We have already estimated $P_n(\mu)$ on $\Omega\setminus\Omega_{n,\eps}$, where $\Omega_{n,\eps}$ is a polyhedral subset of $\Omega$ which, like $\omega_\eps$, has an $O(\eps)$ Lebesgue measure.
We already know that $P_n(\mu_n)_{|\Omega\setminus\Omega_{n,\eps}}\in H^1(\Omega\setminus\Omega_{n,\eps};\bbR^3)$ and $P_n(u_n)_{|\omega_\eps} \rightharpoonup {m_\infty}_{|\Omega\setminus\Omega_{n,\eps}}$ in the sense of $H^1(\Omega\setminus\Omega_{n,\eps};\bbR^3)$, therefore
\begin{eqnarray*}
\int_{\Omega\setminus\Omega_{n,\eps}} |\nabla m_\infty(x)|^2\ \d x  
\leq \liminf_{n\to\infty} \int_{\Omega\setminus\Omega_{n,\eps}} |\nabla P_n(\mu_n)(x)|^2\ \d x 
\leq C \eps.
\end{eqnarray*}

\begin{eqnarray*}
\liminf_{n\to\infty} \cE_{n,\rmex}(\mu_n) 
& = & \liminf_{n\to\infty} (\cE^{\eps}_{n,\rmex}(\mu_{n}) + (\cE_{n,\rmex}(\mu_{n})-\cE^{\eps}_{n,\rmex}(\mu_{n}))) \\
& \geq & \liminf_{n\to\infty} (\cE^{\eps}_{n,\rmex}(\mu_{n})) 
+ \liminf_{n\to\infty} (\cE_{n,\rmex}(\mu_{n})-\cE^{\eps}_{n,\rmex}(\mu_{n})) \\
& \geq & \cE^{\eps}_{\infty,\rmex}(m_\infty) 
+  \liminf_{n\to\infty} (\cE_{n,\rmex}(\mu_{n})-\cE^{\eps}_{n,\rmex}(\mu_{n})) \\
& \geq & \cE_{\infty,\rmex}(m_\infty) 
+ (\cE^{\eps}_{\infty,\rmex}(m_\infty) - \cE_{\infty,\rmex}(m_\infty)) \\
&& + \liminf_{n\to\infty} (\cE_{n,\rmex}(\mu_{n})-\cE^{\eps}_{n,\rmex}(\mu_{n})) \\
& \geq &  \cE_{\infty,\rmex}(m_\infty) - (c+c_1)\eps.
\end{eqnarray*}
Since this holds for all $\eps >0$, we finally deduce that
\begin{equation*}
\liminf_{n\to\infty} \cE_{n,\rmex}(\mu_n) \geq \cE_{\infty,\rmex}(m_\infty).
\end{equation*}

\section{Other energy contributions}
\label{sec_total}

\subsection{Magnetostatics: demagnetizing energy}

We can define a mapping $h_\rmd: L^2(\bbR^3;\bbR^3) \mapsto L^2(\bbR^3;\bbR^3)$ by: for all $u \in L^2(\bbR^3;\bbR^3)$, $h_\rmd(u)$ is solution in the sense of distributions to
\begin{equation*}
\left \{ \begin{array}{rcl} 
\rot h_\rmd(u) & = & 0, \\  
\div h_\rmd(u) & = & - \div u. 
\end{array} \right.
\end{equation*}
When $u$ is the magnetization, $h_\rmd(u)$ is the demagnetizing field. Its energy is
\begin{equation*}
\cE_\rmd(u) = \frac{\mu_0}{2} \|h_\rmd(u)\|_{L^2(\bbR^3;\bbR^3)}.
\end{equation*}
For $u\in L^2(\Omega;\bbR^3)$, we denote by $\tilde{u}$ the $L^2(\bbR^3;\bbR^3)$ function which equals $u$ inside $\Omega$, and $0$ outside $\Omega$.
Hence for a spin distribution $\mu_n \in (\bbS^2)^{\cL_{n,\Omega}}$, the 
demagnetizing energy is defined by
\begin{equation*}
\cE_{n,\rmd}(\mu_n) = \frac{\mu_0}2 \left\|h_\rmd \Big( \widetilde{P_n(\mu_n)}\Big)\right\|_{L^2(\bbR^3;\bbR^3)}.
\end{equation*}

\subsection{External energy}

The external energy, or Zeeman contribution, models the influence of an external magnetic field on the magnetization. Given
such a field $h_\rmZ$ in $C^0(\bbR^3;\bbR^3)$, for all $u \in L^2(\bbR^3;\bbR^3)$ supported in $\Omega$, we set
\begin{equation*}
\cE_\rmZ(u) =-\int_{\Omega} h_\rmZ\cdot u\ \rmd x,
\end{equation*}
this energy in maximized when $u$ is almost everywhere in $\Omega$ in the direction of $h_\rmZ$.\\
At the micro-scale, for a spin distribution $\mu_n \in (\bbS^2)^{\cL_{n,\Omega}}$, we set
\begin{equation*}
\cE_{n,\rmZ}(\mu_n) =-\left(\frac{a}{n}\right)^3\sum_{x\in\cL_{n,\Omega}} h_\rmZ(x)\cdot \mu_{n,x}.
\end{equation*}

\subsection{Total energy}

We define the total energy summing up the exchange, demagnetizing, and external energies both in the lattice context:
\begin{equation*}
\forall \mu_n \in (\bbS^2)^{\cL_{n,\Omega}},\ \cE_n(\mu_n) = \cE_{n,\rmex}(\mu_n)+ \cE_{n,\rmd}(\mu_n) + \cE_{n,\rmZ}(\mu_n),
\end{equation*}
and the limit continuous one:
\begin{equation*}
\forall u \in H^1(\Omega;\bbR^3),\ \cE_\infty(u) =  \cE_{\infty,\rmex}(u)+ \cE_{\infty,\rmd}(u) + \cE_\rmZ(u).
\end{equation*}

\begin{Th}
\label{th_total}
Let $(\mu_n)_{n\in\bbN^*}\in\gH$. In the sense of the topology defined by Definition \ref{def_converge}
\begin{equation*}
\cE_{n} \xrightarrow[n\to\infty]{\Gamma} \cE_\infty.
\end{equation*}
Moreover, this convergence is compatible with the unit norm constraint.
\end{Th}

\begin{proof}
We have already shown that 
\begin{equation*}
\cE_{n,\rmex} \xrightarrow[n\to\infty]{\Gamma} \cE_{\infty,\rmex}.
\end{equation*} 

Moreover $h_\rmd: L^2(\bbR^3;\bbR^3) \mapsto L^2(\bbR^3;\bbR^3)$ is linear and continuous.
We therefore choose $(\mu_n)_{n\in\bbN^*} \in \gH$ such that $\displaystyle \mu_n \xrightarrow[n\to\infty]{} \mu \in H^1(\Omega;\bbR^3) $. 
This means that the sequence $(P_n(\mu_n))$ is weakly convergent in $H^1(\Omega;\bbR^3)$ to $\mu$. 
Hence 
\begin{equation*}
P_n(\mu_n) \xrightarrow[n\to\infty]{L^2(\Omega;\bbR^3)} \mu
\textrm{ and }
\widetilde{P_n(\mu_n)} \xrightarrow[n\to\infty]{L^2(\bbR^3;\bbR^3)} \tilde{\mu}.
\end{equation*}
In particular $\cE_{n,\rmd}(\mu_n) \xrightarrow[n\to\infty]{} \cE_{\infty,\rmd}(\mu)$.\\

Besides if $\mu \in H^1(\Omega;\bbR^3)$, we know that $p_n(\mu) \xrightarrow[n\to\infty]{} \mu$ in $H^1(\Omega;\bbR^3)$, which ends the proof of 
\begin{equation*}
\cE_{n,\rmd} \xrightarrow[n\to\infty]{\Gamma} \cE_{\infty,\rmd}.
\end{equation*}
Last, we notice that
\begin{equation*}
\cE_{n,\rmZ}(\mu_n)=-\left(\frac{a}{n}\right)^3\sum_{x\in\cL_{n,\Omega}} h_\rmZ(x)\cdot \mu_{n,x} = -\left(\frac{a}{n}\right)^3\sum_{x\in\cL_{n,\Omega}} h_\rmZ(x)\cdot P_n(\mu_n)(x),
\end{equation*}
which can be re-written as the approximation of $-\int_{\Omega} h_\rmZ(x)\cdot P_n(\mu_n)(x) \rmd x$ thanks to piecewise approximations. Then, using the regularity of $h_\rmZ$ and the convergence of $\mu_n$ in $L^2(\Omega;\bbR^3)$ toward $\mu$, we prove that
\begin{equation*}
\cE_{n,\rmZ}(\mu_n) \xrightarrow[n\to\infty]{} \cE_\rmZ(\mu).
\end{equation*}
\end{proof}

\section{Conclusion}

In this paper, we prove a $\Gamma$-convergence result from a discrete description of ferromagnetic materials at the microscopic scale to the continuous one. This result has been shown thanks to a rigidity hypothesis on the lattice of magnetic moments. This modeling hypothesis is based on the Heisenberg interaction phenomenon and could be justified by a time multi-scale study. The new hypothesis would take into account the speed of the Heisenberg relaxation compared to the Larmor precession process.

The results in this paper are the seed in order to address the micro--mesoscopic limit for dynamic processes to be able to better understand the
dissipation phenomena involved in the mesoscopic Landau--Lifchitz system.

\end{document}